\documentclass[12pt]{amsart}
\usepackage{amsmath,latexsym,amscd,amsbsy,amssymb,amsfonts,amsthm,fleqn,leqno,
euscript, graphicx, texdraw, pb-diagram}

\numberwithin{equation}{section}
\newtheorem{thm}{Theorem}[section]
\newtheorem{pro}[thm]{Proposition}
\newtheorem{lem}[thm]{Lemma}
\newtheorem{cor}[thm]{Corollary}

\newtheorem*{rem}{Remark}

\theoremstyle{definition}

\theoremstyle{remark}

\numberwithin{equation}{section}

\begin{document}
%%%%%%%%%%% Begin Topmatter %%%%%%%%%%%%%%%%%

\title[Local cohomological properties of homogeneous ANR compacta]
{Local cohomological properties of homogeneous ANR compacta}

\author{V. Valov}
\address{Department of Computer Science and Mathematics,
Nipissing University, 100 College Drive, P.O. Box 5002, North Bay,
ON, P1B 8L7, Canada} \email{veskov@nipissingu.ca}

\thanks{The author was partially supported by NSERC
Grant 261914-13.}


\begin{abstract}
In accordance with the Bing-Borsuk conjecture \cite{bb}, we show that if $X$ is an $n$-dimensional homogeneous metric $ANR$ continuum and $x\in X$, then there is a local basis at $x$ consisting of connected open sets $U$ such that the cohomological properties of $\overline U$ and $bdU$ are similar to the
properties of the closed ball $\mathbb B^n\subset\mathbb R^n$ and its boundary $\mathbb S^{n-1}$. We also prove that a metric $ANR$ compactum $X$ of dimension $n$ is dimensionally full-valued if and only if the group $H_n(X,X\setminus x;\mathbb Z)$ is not trivial for  some $x\in X$, where $\mathbb Z$ is the group of integers. This implies that every $3$-dimensional homogeneous metric $ANR$ compactum is dimensionally full-valued.
\end{abstract}
\maketitle\markboth{}{Homogeneous $ANR$}
%%%%%%%%%% End topmatter %%%%%%%%%%%%%%%%%%%%%

%%%%%%%%%%%%%%%%%%%%%%%%%%%%%%%%%%%%%%%%%%%%%%%%%%%%%%%%%%%%%%%%
%%%%%%%%%%%%%%%%%%%%%%%%%%%%%%%%%%%%%%%%%%%%%%%%%%%%%%%%%%%%%%%%

%%%%%%%%%%%%%%%%%% TABLE OF CONTENT %%%%%%%%%%%%%%%%%%%%%%%%%%%%%%%%%%

%\tableofcontents

%%%%%%%%%%%%%%%%%%%%%%%%%%%%%%%%%%%%%%%%%%%%%%%%%%%%%%%%%%%%%%%%%%%
%%%%%%%%%%%%%%%%%%%%%%%%%%%%%%%%%%%%%%%%%%%%%%%%%%%%%%%%%%%%%%%%%%%%%%

\section{Introduction}
The Bing-Borsuk conjecture \cite{bb} asserts that a homogeneous Euclidean neighborhood retract is a topological manifold. In accordance with that conjecture,
we show that the local cohomological structure of any $n$-dimensional homogeneous metric $ANR$ continuum is similar to the local structure of
$\mathbb R^n$, see Theorem 1.1 below. We also establish conditions for a metric $ANR$ compactum $X$ to satisfy the equality $\dim(X\times Y)=\dim X+\dim Y$ for all compact metric spaces $Y$ (any such $X$ is said to be {\em dimensionally full-valued}). It follows from these conditions that every 3-dimensional homogeneous $ANR$ compactum is dimensionally full-valued (Corollary 1.5), thus providing a partial answer to one of the problems
accompanying the Bing-Borsuk conjecture (whether homogeneous metric $ANR$'s are dimensionally full-valued).

Everywhere in this paper by a space we mean a homogeneous metric $ANR$ continuum $X$  with $\dim_GX=n$, where $n\geq 2$ and $G$ is a fixed countable abelian group or a principal ideal domain (PID) with unity.
Reduced \v{C}ech homology $H_n(X;G)$ and cohomology groups $H^n(X;G)$ with coefficient from $G$ are considered everywhere below.
Let us recall that for any abelian group $G$ the cohomology groups $H^n(X;G)$, $n\geq 2$, are isomorphic to the groups $[X,K(G,n)]$ of pointed homotopy classes of maps from $X$ to $K(G,n)$, where $K(G,n)$ is the Eilenberg-MacLane space of type $(G,n)$, see \cite{sp}. The cohomological dimension $\dim_GX$ is the largest integer $m$ such that there exists a closed set $A\subset X$ with $H^m(X,A;G)\neq 0$. Equivalently, $\dim_GX\leq n$ iff every map $f\colon A\to K(G,n)$ can be extended to a map $\tilde f\colon X\to K(G,n)$.

Suppose $(K,A)$ is a pair of closed subsets of a space $X$ with $A\subset K$. Following \cite{bb}, we say that
$K$ is an {\em $n$-homology membrane spanned on $A$ for an element $\gamma\in H_n(A;G)$} provided $\gamma$ is homologous to
zero in $K$, but not homologous to zero in any proper closed subset of $K$ containing $A$. Similarly, $K$ is said to be
an {\em $n$-cohomology membrane spanned on $A$ for an element $\gamma\in H^n(A;G)$} if $\gamma$ is not extendable over $K$,
but it is extendable over every proper closed subset of $K$ containing $A$. Here, $\gamma\in H^n(A;G)$ is not extendable over $K$ means that
$\gamma$ is not contained in the image $j_{K,A}^n\big(H^{n}(K;G)\big)$, where $j_{K,A}^n:H^{n}(K;G)\to H^{n}(A;G)$ is the
homomorphism generated by the inclusion $A\hookrightarrow K$.

We note the following simple fact, which will be used in this paper and follows from Zorn's lemma and the continuity of \v{C}ech cohomology \cite{sp}: {\em If $A$ is a closed subset of a compact space $X$ and $\gamma$ is an element  of
$H^{n}(A;G)$ not extendable over $X$, then there exists an $n$-cohomology membrane for $\gamma$ spanned on $A$}.

We also say that a closed set $A\subset X$ is  {\em a cohomological carrier}  of a non-zero element $\alpha\in H^n(A;G)$ if $j_{A,B}^n(\alpha)=0$ for every proper closed subset $B\subset A$. If $H^n(A;G)\neq 0$, but $H^n(B;G)=0$ for every closed proper subset $B\subset A$, then $A$ is called an {\em $(n,G)$-bubble}.

\begin{thm}
Let $X$ be a homogeneous metric $ANR$ continuum with $\dim_GX=n$, where $G$ is a countable PID with unity and $n\geq 2$.
Then every point $x$ of $X$ has a basis $\mathcal B_x$ of open sets $U\subset X$ satisfying the following conditions:
\begin{itemize}
\item[(1)] $int\overline U=U$ and the complement of $bdU$ has two components, one of which is $U$;
\item[(2)] $H^{n-1}(\overline U;G)=0$ and $\overline{U}$ is an $(n-1)$-cohomology membrane spanned on $bdU$ for any non-zero $\gamma\in H^{n-1}(bdU;G)$;
\item[(3)] $bdU$ is an $(n-1,G)$-bubble and $H^{n-1}(bdU;G)$ is a finitely generated $G$-module.
\end{itemize}
\end{thm}
The restriction $n\geq 2$ in Theorem 1.1 is needed because of Lemma 2.8, which is used in the proof of Theorem 1.1.

\begin{rem}
Condition $(1)$ from Theorem $1.1$ implies that $\dim_GbdU=n-1$, see \cite{kktv}.
\end{rem}

\begin{thm}
Let $X$ be as in Theorem $1.1$ and $G$ be a countable group. If a closed subset $K\subset X$ is an $(n-1)$-cohomology membrane spanned on $A$ for some closed set $A\subset K$ and some $\gamma\in H^{n-1}(A;G)$, then $(K\setminus A)\cap\overline{X\setminus K}=\varnothing$.
\end{thm}

\begin{cor}
In the setting of Theorem $1.2$, if $U\subset X$ is open and  $f:U\to X$ is an injective map, then  $f(U)$ is open in $X$.
\end{cor}

We already mentioned that a compactum $X$ is  dimensionally full-valued if $\dim (X\times Y)=\dim X+\dim Y$ for any compact metric space $Y$, or equivalently,
$\dim_GX=\dim_{\mathbb Z}X$ for any abelian group $G$.  Recent work of Bryant \cite{bry} was believed to provide a positive answer to the question whether any homogeneous metric $ANR$ is dimensionally full-valued, but Bryant discovered  a gap in the proof of one of the theorems from \cite{bry}. The question whether $\dim (X\times Y)=\dim X+\dim Y$ if both $X$ and $Y$ are homogeneous compact $ANR$s was raised in \cite{clqr} and \cite{f}. Theorem 1.4 below provides some  necessary and sufficient conditions for $ANR$ spaces to be dimensionally full-valued.

\begin{thm}
The following conditions are equivalent for any metric $ANR$-compactum $X$ of dimension $\dim X=n$:
\begin{itemize}
\item[(1)] $X$ is dimensionally full-valued;
\item[(2)] There is a point $x\in X$ with $H_{n}(X,X\setminus x;\mathbb Z)\neq 0$;
\item[(3)] $\dim_{\mathbb S^1}X=n$.
\end{itemize}
\end{thm}

\begin{cor}
Every homogeneous metric $ANR$ compactum $X$ with $\dim X=3$ is dimensionally full-valued.
\end{cor}
\section{Some preliminary results}

In this section, if not stated otherwise, $G$ is a countable abelian group and $X$ denotes a homogeneous metric $ANR$ continuum with $\dim_GX=n$, $n\geq 2$.
If $H^n(X;G)\neq 0$, then $H^n(B;G)=0$ for all proper closed subsets $B$ of $X$, see \cite{vv}. Obviously, this is true when  $H^n(X;G)=0$. Therefore, all proper closed subsets of $X$ have trivial $n$-cohomology groups.

We begin with the following analogue of Theorem 8.1 from \cite{bb} (it is here that the countability of $G$ is used).
\begin{pro}\label{membrane}
Theorem $1.2$ holds under the additional assumption that $K$ is contractible in a proper subset of $X$.
\end{pro}

\begin{proof}
According to the duality between homology and cohomology for countable groups \cite[viii 4G)]{hw}, for any compact metric space $Y$ the groups $H_{n-1}(Y,G^*)$ and  $H^{n-1}(Y;G)^*$ are isomorphic, where $G^*$ and $H^{n-1}(Y;G)^*$ denote the character groups of $G$ and $H^{n-1}(Y;G)$, respectively. Here both $H^{n-1}(Y;G)$ and $G$ are considered as discrete groups. Using this duality, we can show that $K$ is an $(n-1)$-homology membrane for some $\beta\in H_{n-1}(A,G^*)$ spanned on $A$. Indeed, consider the homomorphism
$j_{K,A}^{n-1}:H^{n-1}(K;G)\to H^{n-1}(A;G)$. Since $\gamma$ is not extendable over $K$,
$\gamma\not\in G_A=j_{K,A}^{n-1}\big(H^{n-1}(K;G)\big)$. Considering $H^{n-1}(A;G)$ as a discrete group, we can find a character
$\beta\colon H^{n-1}(A;G)\to\mathbb S^1$ such that $\beta(\gamma)\neq e$ and $\beta(G_A)=e$, where $e$ is the unit of $\mathbb S^1$. On the other hand, $\gamma$ is extendable over every proper closed subset $B$ of $K$ which contains $A$. Therefore, $\gamma$ is contained in the image of $j_{B,A}^{n-1}:H^{n-1}(B;G)\to H^{n-1}(A;G)$ for any such $B$. Then the composition $j_{K,A}^{n-1}\circ\beta$ is the trivial character
of $H^{n-1}(K;G)$, while the composition $j_{B,A}^{n-1}\circ\beta$ is non-trivial for any proper closed subset $B$ of $K$ containing $A$. So, $\beta$ is homologous to zero in $K$, but not homologous to zero in any proper closed subset of $K$ containing $A$. Hence,
$K$ is an $(n-1)$-homology membrane for $\beta$ spanned on $A$.

Now, assume that $(K\setminus A)\cap\overline{X\setminus K}\neq\varnothing$.
Then following the proof of Theorem 16.1 from \cite{bo} (see also \cite[Theorem 8.1]{bb}), we can find a proper closed subset $\Gamma$ of $X$ and a non-zero element $\alpha\in H_{n}(\Gamma,G^*)$. This means that $H^n(\Gamma;G)\neq 0$, a contradiction.
\end{proof}

Since the Bing-Borsuk result used in the proof of Proposition 2.1 was established for locally homogeneous spaces, Proposition 2.1 remains valid for locally homogeneous spaces $X$ such that $H^n(A;G)$ is trivial for any proper closed subset $A\subset X$.

\begin{cor}
Let $A\subset X$ be a closed subset and $K$ an $(n-1)$-cohomology membrane for some $\gamma\in H^{n-1}(A;G)$ spanned on $A$. Then $K\setminus A$ is connected. If, in addition, $K$ is contractible in a proper subset of $X$, then $K\setminus A$ is an open subset of $X$.
\end{cor}

\begin{proof}
 Suppose $K\setminus A$ is the union of two non-empty, disjoint open sets $U$ and $V$. Then  $K\setminus U$ and $K\setminus V$ are closed proper subsets of $K$ such that $(K\setminus U)\cap (K\setminus V)\subset A$. Hence, $\gamma$ is extendable over each of these sets and, because $A$ contains their common part, $\gamma$ is extendable over $K$. The last conclusion contradicts the fact that $K$ is $(n-1)$-cohomology membrane for $\gamma$.

If $K$ is contractible in a proper subset of $X$, then $(K\setminus A)\cap\overline{X\setminus K}=\varnothing$ (see Proposition 2.1). Hence, $K\setminus A$ is open in $X$.
\end{proof}

\begin{cor}
For any closed set $Z\subset X$ one has $\dim_GZ=n$ if and only if $Z$ has a non-empty interior in $X$.
\end{cor}

\begin{proof}
This was established by Seidel in \cite{s} for the covering dimension. His arguments can be modified for $\dim_G$. If $\dim_GZ=n$, we may assume that $Z$ is contractible in a proper  subset of $X$ (this can be done because $X$ is locally contractible and $\dim_G$ satisfies the countable sum theorem). Since $\dim_GZ=n$, there exists a closed set $A\subset Z$ such that $H^{n}(Z,A;G)\neq 0$. On the other hand,
$H^{n}(Z;G)=0$ (as a proper closed subset of $X$). So, according to the exact sequence
{ $$
\begin{CD}
H^{n-1}(Z;G)@>{{j_{Z,A}^{n-1}}}>>H^{n-1}(A;G)@>{{\delta}}>>H^{n}(Z,A;G)\rightarrow 0
\end{CD}
$$}\\
there exists $\gamma\in H^{n-1}(A;G)$ not extendable over $Z$.
Hence, as it was noted above, we can find a closed subset $K$ of $Z$ such that $K$ is an $(n-1)$-cohomological membrane for $\gamma$ spanned on $A$. So,  $K\setminus A$ is open in $X$ (by Corollary 2.2) and $K\setminus A\subset Z$.

If $Z$ has a non-empty interior, then it contains an open set $U$ in $X$ with $\dim_GU=n$. So, $\dim_GZ=n$
\end{proof}

\begin{lem} Let a closed set $F\subset X$ with $H^{n-1}(F;G)\neq 0$ be contractible in an open set $U\subset X$. If $\overline U$ is contractible in a proper subset of $X$, then $F$ separates $\overline{W}$ for any open set $W\subset X$ containing $U$.
\end{lem}

\begin{proof}
Indeed, there is a closed set $P$ in $X$ such that $P\subset U$ and $F$ is contractible in $P$. Then
any non-zero element $\gamma\in H^{n-1}(F;G)$ is not extendable over $P$ (otherwise $\gamma$, considered as a map from $F$ to  $K(G,n-1)$, would be homotopic to a constant because $F$ is contractible in $P$). This yields the existence of an $(n-1)$-cohomology membrane $K_\gamma\subset P$ for $\gamma$ spanned on $F$. Because $\overline U$ is contractible in a proper subset of $X$, so is $K_\gamma$. Hence, by Proposition 2.1, $(K_\gamma\setminus F)\cap\overline{X\setminus K_\gamma}=\varnothing$. The last equality implies that $F$ separates any $\overline{W}$ such that $W\subset X$ is open and contains $U$.
\end{proof}

\begin{lem}
Suppose $U\subset X$ is open and $P\subsetneqq X$ is closed such that $\overline{U}\subsetneqq P$ and $H^{n-1}(bdU;G)$ contains elements not extendable over $\overline{U}$. Then, there exists  $\gamma\in H^{n-1}(bdU;G)\setminus L$ extendable over
$P\setminus V$, where $V=int(\overline{U})$ and $L=j_{\overline{U},bdU}^{n-1}\big(H^{n-1}(\overline{U};G)\big)$. Moreover, if $L=0$, then every $\gamma\in H^{n-1}(bdU;G)$ is extendable over $P\setminus V$.
\end{lem}

\begin{proof}
Indeed, since $H^{n-1}(bdU;G)$ contains elements not extendable over $\overline{U}$,
$L$ is a  proper subgroup of $H^{n-1}(bdU;G)$. Consider the homomorphism $j_{P\setminus V,bdU}^{n-1}\colon H^{n-1}(P\setminus V;G)\to H^{n-1}(bdU;G)$. It suffices to show that the image of $H^{n-1}(P\setminus V;G)$ under $j_{P\setminus V,bdU}^{n-1}$ is not contained in $L$. To this end, suppose
$j_{P\setminus V,bdU}^{n-1}\big(H^{n-1}(P\setminus V;G)\big)\subset L$.  Consider the Mayer-Vietoris exact sequence, where $A=P\setminus V$ and $\varphi(\gamma_1,\gamma_2)=j_{A,bdU}^{n-1}(\gamma_2)-j_{\overline{U},bdU}^{n-1}(\gamma_1)$ for $\gamma_1\in H^{n-1}(\overline{U};G)$, $\gamma_2\in H^{n-1}(A;G)$:
{ $$
\begin{CD}
H^{n-1}(\overline{U};G)\oplus H^{n-1}(A;G)@>{{\varphi}}>>H^{n-1}(bdU;G)@>{{\triangle}}>>H^n(P;G)\rightarrow
\end{CD}
$$}
Obviously, $L_U=\varphi\big(H^{n-1}(\overline{U};G)\oplus H^{n-1}(A;G)\big)\subset L$. Consequently,
 any $\gamma\in H^{n-1}(bdU;G)\setminus L$ is not contained in $L_U$. Hence, $\triangle(\gamma)\neq 0$ for all $\gamma\in H^{n-1}(bdU;G)\setminus L$. So, $H^n(P;G)\neq 0$, a contradiction (recall that the $n$-th cohomology groups of all proper closed sets in $X$ are trivial).

If $L=0$, then $j_{\overline{U},bdU}^{n-1}(\gamma_1)=0$ for all $\gamma_1\in H^{n-1}(\overline{U};G)$, so $\varphi(\gamma_1,\gamma_2)=j_{A,bdU}^{n-1}(\gamma_2)$. Since
$\triangle(H^{n-1}(bdU;G))=0$, we obtain that for any element $\gamma\in H^{n-1}(bdU;G)$ there exist $\gamma_1\in H^{n-1}(\overline{U};G)$ and $\gamma_2\in H^{n-1}(A;G)$ such that $\varphi(\gamma_1,\gamma_2)=\gamma$. Hence, $\gamma=j_{A,bdU}^{n-1}(\gamma_2)$, which means that $\gamma$ is extendable over $A$. This completes the proof.
\end{proof}

\begin{lem}
  If $U\subset X$ is a connected open set and $\overline U$ is contractible in a proper subset of $X$, then $\overline U$ is  an $(n-1)$-cohomology membrane spanned on $bdU$ for every $\gamma\in H^{n-1}(bdU;G)$ not extendable over $\overline{U}$.
\end{lem}

\begin{proof}
Observe first that $U$ is dense in $V=int(\overline U)$, so $V$ is also connected.
Let $\gamma$ be an element of $H^{n-1}(bdU;G)$ not extendable over $\overline{U}$.
Then there exists a closed subset $K\subset\overline{U}$ such that $K$ is an $(n-1)$-cohomology membrane for $\gamma$ spanned on $bdU$. Since $K$ is contractible in a proper subset of $X$ (as a subset of $\overline{U}$), by Proposition \ref{membrane}, $(K\setminus bdU)\cap\overline{X\setminus K}=\varnothing$. Hence, $K\setminus bdU$ is open in $X$. This implies that $K=\overline{U}$, otherwise $V$ would be the union of the non-empty disjoint open sets $V\setminus K$ and $(K\setminus bdU)\cap V$. Therefore, $\overline{U}$ is an $(n-1)$-cohomology membrane spanned on $bdU$ for $\gamma$.
\end{proof}

The last two statements of this section (Lemma 2.7 and Lemma 2.8) hold for arbitrary compactum $X$.

\begin{lem}
Let $X$ be an arbitrary compactum and $A\subset X$ be a carrier for a non-zero element $\gamma\in H^{n-1}(A;G)$ with $\dim_GA\leq n-1$, $n\geq 2$. Then $A$ is connected.
\end{lem}

\begin{proof}
Suppose $A$ is not connected, so $A$ is the union of two closed disjoint non-empty sets $A_1$ and $A_2$. Then $H^{n-1}(A;G)$ is isomorphic to the direct sum $H^{n-1}(A_1;G)\bigoplus H^{n-1}(A_2;G)$ and $\gamma$ is identified with the pair $(\gamma_1,\gamma_2)$, where $\gamma_i=j_{A,A_i}^{n-1}(\gamma)$, $i=1,2$.
 Because $A$ is a carrier of $\gamma$ and $A_i$ are proper
 closed non-empty subsets of $A$, $\gamma_1=\gamma_2=0$. So, $\gamma=0$, a contradiction.
\end{proof}

Since that $\dim_GA=0$ is equivalent with $\dim A=0$, Lemma 2.8 is not valid for $n=1$. For example, if $A$ consists of two different points, then there exists a non-trivial element of $\gamma\in H^{0}(A;\mathbb Z)$ such that $A$ is a carrier of $\gamma$.

Suppose $G$ is a group (resp., a ring). Let $F\subset Z\subset X$ be compact sets. We say that $F$ is an {\em $(n-1,G)$-bubble with respect to a subgroup (resp., a submodule) $L\subset H^{n-1}(Z;G)$} if the group (resp., the submodule) $j_{Z,F}^{n-1}(L)\subset H^{n-1}(F;G)$ is non-trivial, but $j_{Z,B}^{n-1}(L)\subset H^{n-1}(B;G)$ is trivial for any closed proper subset $B\subset F$.

\begin{lem} Let $G$ be a group $($resp., a ring$)$. If $Z$ is a closed subset of an arbitrary compactum $X$ and $L\subset H^{n-1}(Z;G)$ is a non-trivial and finitely generated subgroup $($resp., a submodule$)$, then $Z$ contains a non-empty closed subset $F$ such that $F$ is an $(n-1,G)$-bubble with respect to $L$.
\end{lem}

\begin{proof}
  If $L$ has one generator $\gamma$, we just take a closed set $F\subset Z$, which is a carrier for $\gamma$. Then $\beta=j_{Z,F}^{n-1}(\gamma)$ and $\beta_B=j_{Z,B}^{n-1}(\gamma)$ are  generators, respectively, of $j_{Z,F}^{n-1}(L)\subset H^{n-1}(F;G)$ and $j_{Z,B}^{n-1}(L)\subset H^{n-1}(B;G)$ for any closed set $B\subset Z$. So, $j_{Z,B}^{n-1}(L)=0$ for every proper closed subset $B$ of $F$ because $j_{Z,B}^{n-1}(\gamma)=j_{F,B}^{n-1}(\beta)=0$. Hence, $F$ is an $(n-1,G)$-bubble with respect to $L$. Suppose our lemma is true for any such set $Z$  and a subgroup (resp., a submodule)
 $L\subset H^{n-1}(Z;G)$ with $\leq k$ generators. In case $L$ has $k+1$ generators
$\gamma_1,..,\gamma_k,\gamma_{k+1}$, we first take a closed non-empty set $F_1\subset Z$, which is a carrier for $\gamma_1$. So, $j_{Z,B}^{n-1}(\gamma_1)=0$ for any proper closed subset $B$ of $F_1$. If $H^{n-1}(B;G)=0$ for all closed $B\subsetneqq F_1$, then $F_1$ is as required. If $j_{Z,B^*}^{n-1}(L)\neq 0$ for some closed proper set $B^*\subset F_1$, then $j_{Z,B^*}^{n-1}(L)$ is generated by the set $\{j_{Z,B^*}^{n-1}(\gamma_i):i=2,3,..,k+1\}$.  According to our inductive assumption, there exists a closed non-empty set $F\subset B^*$ being an $(n-1,G)$-bubble in $B^*$ with respect to $j_{Z,B^*}^{n-1}(L)$. Then $F$ is an $(n-1,G)$-bubble in $Z$ with respect to $L$.
\end{proof}

\section{Proof of Theorems 1.1, 1.2 and Corollary 1.3}

In this section $X$ continues to be as in Section 2, but $G$ is assumed to a countable PID (the last condition is used in the proof of Claim 1).

%We consider the following properties of a space $X$ and open subsets $U\subset X$.
%\begin{itemize}
%\item[$M_1(n)$:] $\dim_GbdU=n-1$, $H^{n-1}(bdU;G)\neq 0$ and there exists $\gamma\in H^{n-1}(bdU;G)$ not extendable over $\overline{U}$.
%\item[$M_2(n)$:] $H^{n-1}(\overline U;G)=0$ and $\overline{U}$ is an $(n-1)$-cohomology membrane spanned on $bdU$ for any non-zero $\gamma\in H^{n-1}(bdU;G)$.% not extendable over $\overline U$.
%\item[$M_3(n)$:] $H^{n-1}(bdU;G)$ is finitely generated $G$-module and $bdU$ is an $(n-1,G)$-bubble.
%\end{itemize}

\textit{Proof of Theorem $1.1$.}
As in the proof of Proposition 2.1, we may suppose that $X$ is connected and $H^n(C;G)=0$ for any closed proper subset $C$ of $X$. Moreover, we equip $X$ with a convex metric $d$ generating its topology (such a metric exists, see \cite{bi}).
According to \cite[Theorem 2]{ku}, there exists a closed subset $Y\subset X$ with $\dim_GY=n$ and a dense open subset $D$ of $Y$ satisfying the following property: any $y\in D$ has sufficiently small neighborhoods $U_y$ in $Y$ such that the homomorphism $j_{\overline U_y,bd_Y\overline U_y}^{n-1}$ is not surjective (here $bd_Y\overline U_y$ denotes the boundary of $\overline U_y$ in $Y$).
Because $Y$ has a non-empty interior in $X$ (by Corollary 2.3), there exists a point $x\in int(Y)\cap D$, a connected open neighborhood $W_{x}$ of $x$ in $X$ and an element $\alpha_x\in H^{n-1}(bd\overline W_x;G)$ such that $\alpha_x$ is not extendable over $\overline W_{x}$.  We can suppose that $\overline W_{x}$ is contractible in a proper subset of $X$. So, by Lemma 2.6, $\overline W_x$ is an
$(n-1)$-cohomology membrane for $\alpha_x$ spanned on $bd\overline W_x$.
Because $X$ is homogeneous, it suffices to construct the required base $\mathcal B_x$ at that particular point $x$.
We define
$\mathcal B_x'$ to be the family of all open connected subsets $U\subset X$ containing $x$ such that $U=int(\overline{U})$ and $\overline U$ is contractible in $W_x$. Then $\mathcal B_x'$ is a local base at $x$ and $bdU=bd\overline U$ for all $U\in\mathcal B_x'$

\smallskip
\textit{Claim $1$. Every $U\in\mathcal B_x'$ has the following properties:
\begin{itemize}
\item[(i)] $\overline U$ is an $(n-1)$-cohomology membrane for some element of the group $H^{n-1}(bdU;G)$;
%\item $U$ is connected provided $U=Int(\overline{U})$;
\item[(ii)] the module $L_U=j_{\overline W_{x}\setminus U,bdU}^{n-1}\big(H^{n-1}(\overline W_{x}\setminus U;G)\big)\subset H^{n-1}(bdU;G)$ is non-trivial and finitely generated;
\item[(iii)] the module $H^{n-1}(bdU;G)$ is finitely generated provided the homomorphism $j_{\overline U,bdU}^{n-1}$ is trivial.
%\item[(iv)] $H^{n-1}(bdU;G)\neq 0$ and there exists $\gamma\in H^{n-1}(bdU;G)$ not extendable over $\overline{U}$.
\end{itemize}}

\smallskip
We fix $U\in\mathcal B_x'$ and
a non-zero element $\alpha_x\in H^{n-1}(bd\overline W_x;G)$ such that $\overline W_{x}$ is an $(n-1)$-cohomology membrane for $\alpha_x$ spanned on $bd\overline W_{x}$. Then $\alpha_x$ is not extendable over $\overline W_{x}$ but it is extendable over every closed proper subset of $\overline W_{x}$.  Next, extend $\alpha_x$ to an element
$\widetilde\alpha_x\in H^{n-1}(\overline W_{x}\setminus U;G)$. Obviously,  $bdU\subset\overline W_{x}\setminus U$. Hence, the element $\gamma_U=j_{\overline W_{x}\setminus U,bdU}^{n-1}(\widetilde\alpha_x)\in H^{n-1}(bdU;G)$ is not extendable over $\overline U$ (otherwise $\alpha_x$ would be extendable over $\overline W_{x}$), in particular $\gamma_U\neq 0$. Since $U$ is connected, by Lemma 2.6, $\overline U$ is an $(n-1)$-cohomology membrane for  $\gamma_U$ spanned on $bdU$.

To prove the second item $(ii)$, let $U_0$ be an open subset of $X$ with $\overline U_0\subset U$. Since $\gamma_U\in L_U$ and $\gamma_U\neq 0$, $L_U\neq 0$. For any $\gamma\in L_U$ there are two possibilities: either $\gamma$ is extendable over $\overline U$ or it is not extendable over $\overline U$. In both cases $\gamma$ is extendable over the set $\overline U\setminus U_0$. Indeed, this is clear if $\gamma$ is extendable on $\overline U$.
If $\gamma$ is not extendable over $\overline U$, then $\overline U$ is an $(n-1)$-cohomology membrane for  $\gamma$ spanned on $bdU$ (Lemma 2.6). Consequently, $\gamma$ is extendable over $\overline U\setminus U_0$ because $\overline U\setminus U_0$ is a proper subset of $\overline U$ containing $bdU$.
Hence, every $\gamma\in L_U$ is extendable over the set
$\overline W_x\setminus U_0$, which is closed in $X$ and contains $bdU$ in its interior. Therefore, by \cite[Theorem 17.4 and Corollary 17.5, p.127]{br}, $L_U$ is finitely generated.
If $j_{\overline U,bdU}^{n-1}\big(H^{n-1}(\overline U;G)\big)=0$, then every $\gamma\in H^{n-1}(bdU;G)$ is extendable over
$\overline W_{x}\setminus U$, see Lemma 2.5. Hence, $H^{n-1}(bdU;G)\subset L_U$, and item $(ii)$ yields item $(iii)$.

%\smallskip
 %According to Corollary 2.2, $\dim_GbdU=n-1$ for all such $U\in\mathcal B_x'$. Hence, by Claim 1, each $U\in\mathcal B_x'$ has $M_1(n)$.

\smallskip
Let
$\mathcal B_x''$ be the family of all $U\in\mathcal B_x'$ satisfying  the following condition: $bdU$ contains a continuum $F_U$ such that $X\setminus F_U$ has exactly two components and $F_U$ is an $(n-1,G)$-bubble with respect to the module $L_U$.

\smallskip
\textit{Claim $2$. $\mathcal B_x''$ is a local base at $x$.}

\smallskip
We fix $W_0\in\mathcal B_x'$ and for every $\delta>0$ denote by $B(x,\delta)$ the open ball in $X$ with a center $x$ and a radius $\delta$. There exists
$\varepsilon_x>0$ such that $B(x,\delta)\subset W_0$ for all $\delta\leq\varepsilon_x$. Since $d$ is a convex metric, each $B(x,\delta)$ is a connected open set such that $int(\overline{B(x,\delta)})=B(x,\delta)$. Because $\overline W_0$ is contractible in $W_x$, so is $\overline{B(x,\delta)}$.
Hence, all $U_\delta=B(x,\delta)$, $\delta\leq\varepsilon_x$, belong to $\mathcal B_x'$. Consequently, by Claim 1, the modules
$L_\delta=j_{\overline W_{x}\setminus U_\delta,bdU_\delta}^{n-1}\big(H^{n-1}(\overline W_{x}\setminus U_\delta;G)\big)$ are finitely generated.
Then, by Lemma 2.8, there exists a closed non-empty set $F_\delta\subset bdU_\delta$ with $F_\delta$ being an $(n-1;G)$-bubble with respect to $L_\delta$. Because $F_\delta$ is a carrier for any $\gamma\in L_\delta$, Lemma 2.7 yields that each $F_\delta$ is a continuum.
 Let us show that
the family $\{F_\delta:\delta\leq\varepsilon_x\}$ is uncountable. Since the function $f\colon X\to\mathbb R$, $f(y)=d(x,y)$, is continuous and $W_0$ is connected, $f(W_0)$ is an interval containing $[0,\varepsilon_x]$ and $f^{-1}([0,\varepsilon_x))=B(x,\varepsilon_x)\subset W_0$. So,  $f^{-1}(\delta)=bdU_\delta\neq\varnothing$ for all $\delta\leq\varepsilon_x$. Hence, the family $\{F_\delta:\delta\leq\varepsilon_x\}$ is indeed uncountable
and consist of disjoint continua. Moreover, $H^{n-1}(F_\delta;G)\neq 0$ and, according to Lemma 2.4,  $F_\delta$ separates $X$. So, each $X\setminus F_\delta$ has at least two components. Then, by \cite[Theorem 8]{ch}, there exists $\delta_0\leq\varepsilon_x$ such that $X\setminus F_{\delta_0}$ has exactly two components. Therefore, $U_{\delta_0}=B(x,\delta_0)\in\mathcal B_x''$ and it is contained in $W_0$. This completes the proof of Claim 2.

\smallskip
Now, let $\mathcal B_x$ be the subfamily of all $U\in\mathcal B_x''$ such that $H^{n-1}(bdU;G)\neq 0$ and both $U$ and $X\setminus\overline U$ are connected.

\smallskip
\textit{Claim $3$. $\mathcal B_x$ is a local base at $x$.}

\smallskip
 We take an arbitrary neighborhood $U_0$ of $x$ such that $\overline U_0$ is contractible in $W_x$ and shall construct a member of $\mathcal B_x$ contained in $U_0$. To this end let $\varepsilon=d(x,X\setminus U_0)$. According to the Effros' theorem \cite{e}, there is $\eta>0$ such that if $y,z\in X$ with $d(y,z)<\eta$, then $h(y)=z$ for some homeomorphism $h\colon X\to X$, which is $\varepsilon/2$-close to the identity on $X$. Now, choose a connected neighborhood $W$ of $x$ with  $\overline{W}\subset B(x,\varepsilon/2)$ and $diam(\overline W)<\eta$.
Finally, take $U\in\mathcal B_x''$ such that $\overline{U}$ is contractible in $W$.
  There exists a continuum $F_U\subset bdU$ such that $X\setminus F_U$ has exactly two components and $F_U$ is an $(n-1,G)$-bubble with respect to the module $L_U=j_{\overline W_{x}\setminus U,bdU}^{n-1}\big(H^{n-1}(\overline W_{x}\setminus U;G)\big)$ (see Claim 2). If $F_U=bdU$ we are done, for $U$ is the desired member of $\mathcal B_x$.
 %Indeed, since $X\setminus bdU=U\cup X\setminus\overline U$ and $U$ is connected, then $X\setminus\overline U$ should be also connected (recall that $X\setminus bdU$ has exactly two components).

   Suppose that $F_U$ is a proper subset of $bdU$. Because $F_U$ is an $(n-1,G)$-bubble with respect to $L_U$, $j_{bdU,F_U}^{n-1}(L_U)\neq 0$. Hence, there exists $\gamma\in L_U$ such that $\beta=j_{bdU,F_U}^{n-1}(\gamma)\neq 0$. Because $F_U$ (as a subset of $\overline U$) is contractible in $W$ and
   $\overline W$ (as a subset of $\overline W_x$) is contractible in a proper subset of $X$, we can apply
Lemma 2.4 to conclude that $F_U$ separates
$\overline{W}$. So, $\overline{W}\setminus F_U=V_1\cup V_2$ for some open, non-empty disjoint subsets $V_1,V_2\subset\overline{W}$. Since $U$ is a connected subset of $\overline{W}\setminus F_U$,  $U$ is contained in one of the sets $V_1,V_2$, say $U\subset V_1$. Hence, $F_U\cup\overline V_2\subset\overline W_{x}\setminus U$. Observe that $\gamma\in L_U$ implies $\gamma$ is extendable over $\overline W_{x}\setminus U$. Consequently, $\beta$ is also extendable over $\overline W_{x}\setminus U$, in particular $\beta$ is extendable over $F_U\cup\overline V_2$. On the other hand, $F_U$ (as a subset of $\overline U$) is contractible in $\overline{W}$, so $\beta$ is not extendable over $\overline{W}$ ( otherwise $\beta$ would be zero). Thus, since
$(F_U\cup\overline V_1)\cap (F_U\cup\overline V_2)=F_U$,
$\beta$ is not extendable over $F_U\cup\overline V_1$.
 Let
$\beta'=j_{F_U,F'}^{n-1}(\beta)$, where $F'=\overline{V}_1\cap F_U$
(observe that $F'\neq\varnothing$ because $\overline W$ is connected). If $F'$ is a proper subset of $F_U$, then $\beta'=0$ (recall that $j_{bdU,F'}^{n-1}(\gamma)=\beta'$ and $F_U$ being a carrier for any non-trivial element of $j_{bdU,F_U}^{n-1}(L_U)$ yields $j_{bdU,Q}^{n-1}(L_U)=0$ for any proper closed subset $Q$ of $F_U$). So, $\beta'$ would be extendable over $\overline V_1$, which yields $\beta$ is extendable over $F_U\cup\overline V_1$, a contradiction. Therefore, $F'=F_U\subset\overline{V}_1$ and
$\beta$ is not extendable over $\overline{V}_1$. Consequently, there exists
an $(n-1)$-cohomology membrane $P_\beta\subset\overline{V}_1$ for $\beta$ spanned on $F_U$. By Corollary 2.2, $V=P_\beta\setminus F_U$ is a connected open set in $X$ whose boundary, according to Proposition 2.1, is the set
 $F''=\overline{X\setminus P_\beta}\cap\overline{P_\beta\setminus F_U}\subset F_U$ (we can apply Proposition 2.1 and Corollary 2.2 because $P_\beta$, as a subset of $\overline W_x$, is contractible in a proper subset of $X$).
  As above, using that $\beta$ is not extendable over $P_\beta$ and $j_{bdU,Q}^{n-1}(L_U)=0$ for any proper closed subset $Q\subset F_U$, we can show that $F''=F_U$ and $bd\overline V=F_U$.
 Summarizing the properties of $V$, we have that
 $\overline{V}$ is contractible in $W_x$ (because so is $\overline U_0$),
 $V=int(\overline V)$ (because $F_U=bd\overline V$) and $V$ is connected. Moreover, since $X\setminus F_U$ is the union of the open disjoint non-empty sets $V$ and $X\setminus P_\beta$ such that $V$ is connected and
 $X\setminus F_U$ has exactly two components, $X\setminus\overline V$ is also connected. Because $F_U$ is an $(n-1,G)$-bubble with respect to the non-trivial module $L_U$, $H^{n-1}(bdV;G)\neq 0$. Thus, if $V$ contains $x$, then $V$ is the desired member of $\mathcal B_x$.

  If $V$ does not contain $x$, we take a point $y\in V$ and a homeomorphism $h$ on $X$ such that $h(y)=x$ and $d(z,h(z))<\varepsilon$ for all $z\in X$. Such a homeomorphism exists because $diam(\overline W)<\eta$ and $x,y\in\overline W$.
 Then $h(V)\subset U_0$ (this inclusion follows from the choice of $\varepsilon$ and the fact that $h$ is $\varepsilon$-close to the identity on $X$). So,
$\overline{h(V)}$ is contractible in $W_x$. Since the remaining properties from the definition of $\mathcal B_x$ are invariant under homeomorphisms,
 $h(V)$ is the desired member of $\mathcal B_x$, which completes the proof of Claim 3.

The sets $U\in\mathcal B_x$ satisfy condition $(1)$ from Theorem 1.1 (according to the definition of $\mathcal B_x$). Next claim completes the proof of Theorem 1.1.

\smallskip
\textit{Claim $4$. Every $U\in\mathcal B_x$ satisfies conditions $(2)$ and $(3)$ from Theorem $1.1$.}

\smallskip
Recall that each $\overline U$ is contractible in the set $W_x$ and $\overline W_x$ is contractible in a proper subset of $X$. Then, by Lemma 2.4,
$H^{n-1}(\overline U;G)=0$ because $\overline U$ does not separate $X$. Therefore, every non-trivial element  $\gamma\in H^{n-1}(bdU;G)$ is not extendable over $\overline U$. Consequently, according to Lemma 2.6, $\overline U$ is an $(n-1)$-cohomology membrane for $\gamma$ spanned on $bdU$.
So, $U$ satisfies condition $(2)$.

Since $H^{n-1}(\overline U;G)=0$, the homomorphism $j_{\overline{U},bdU}^{n-1}$ is trivial. Thus, Lemma 2.5 yields
$H^{n-1}(bdU;G)=j_{\overline W_{x}\setminus U,bdU}^{n-1}\big(H^{n-1}(\overline W_{x}\setminus U;G)\big)$ and, by Claim 1(iii), $H^{n-1}(bdU;G)$  is finitely generated. Suppose there exists a proper closed subset $F\subset bd U$ and a non-trivial element $\alpha\in H^{n-1}(F;G)$. Observe that $\alpha$ is not extendable over $\overline U$ because $H^{n-1}(\overline U;G)=0$. Hence, there is an $(n-1)$-cohomology membrane $K_\alpha\subset\overline U$ for $\alpha$ spanned on $F$. Because $\overline U\setminus F$ is connected (recall that $U$ is a dense connected subset of $\overline U\setminus F$) and $K\setminus F$ is both open and closed in $\overline U\setminus F$ (by Corollary 2.2), $K_\alpha=\overline U$. Finally, according to Proposition 2.1, we have $(K_\alpha\setminus F)\cap\overline{X\setminus K_\alpha}=\varnothing$. On the other hand, any point from $bdU\setminus F$ belongs to $(K_\alpha\setminus F)\cap\overline{X\setminus K_\alpha}$, a contradiction. Therefore, $bdU$ is an $(n-1,G)$-bubble and $U$ satisfies condition $(3)$. $\Box$

\smallskip
\textit{Proof of Theorem $1.2$.} If $K=X$, Theorem 1.2 is obviously true. Suppose $K$ is a proper closed subset of $X$, which is an $(n-1)$-cohomology membrane spanned on $A$ for some $\gamma\in H^{n-1}(A;G)$, but there exists a point $a\in(K\setminus A)\cap\overline{X\setminus K}$. Take a neighborhood $U\in\mathcal B_a$  such that $\overline{U}\cap A=\varnothing$. Since $K\setminus U$ is a closed proper subset of $K$ containing $A$, $\gamma$ is extendable over $K\setminus U$. So, there exists $\beta\in H^{n-1}(K\setminus U;G)$ with $j_{K\setminus U,A}^{n-1}(\beta)=\gamma$. Since $K\setminus A$ is connected (see Corollary 2.2), $bdU\cap K\neq\varnothing$.
Then
$\beta_1=j_{K\setminus U,bdU\cap K}^{n-1}(\beta)$ is a non-zero element of $H^{n-1}(bdU\cap K;G)$ (otherwise $\beta_1$ would be extendable over
$\overline{U}\cap K$, and hence, $\gamma$ would be extendable over $K$). Since $\dim_GbdU\leq n-1$, $\beta_1$ is extendable to an element
$\tilde{\beta}_1\in H^{n-1}(bdU;G)$.
%This implies that $bdU=bdU\cap K$ (recall that $bdU$ is an $(n-1,G)$-bubble).
So, $\tilde{\beta}_1$ is a non-zero element of $H^{n-1}(bdU;G)$ and, by Theorem 1.1(2), $\overline{U}$ is an $(n-1)$-cohomology membrane for $\tilde{\beta}_1$ spanned on $bdU$. Then $\overline{U}\cap K\neq\overline{U}$ would yields that $\tilde{\beta}_1$ is extendable over $\overline{U}\cap K$. Hence, $\gamma$ would be extendable over $K$, a contradiction. Thus, $\overline{U}\subset K\setminus A$ which contradicts $a\in\overline{X\setminus K}$. Therefore, $(K\setminus A)\cap\overline{X\setminus K}=\varnothing$.
$\Box$

\textit{Proof of Corollary $1.3$.}
It was shown in \cite{l} and \cite{s} that the cohomology membranes' property from Theorem 1.2 implies the invariance of domain for homogeneous or locally homogeneous $ANR$-spaces $X$ with $\dim X=n$. Similar arguments provide the proof when $\dim_GX=n$.
Take a point $y\in V=f(U)$ and let $x=f^{-1}(y)$.
Choose a connected open set $W\in\mathcal B_x$ such that $\overline{W}\subset U$. Then $\overline{W}$ is an $(n-1)$-cohomology membrane for some $\gamma\in H^{n-1}(bdW;G)$ spanned on $bdW$. Since $f(\overline{W})$ is homeomorphic to $\overline{W}$, it is an $(n-1)$-cohomology membrane for $(f^*)^{-1}(\gamma)\in H^{n-1}(f(bdW);G)$ spanned on $f(bdW)$. Then, by Theorem 1.2, $f(\overline{W})\setminus f(bdW)$ does not intersect
$\overline{X\setminus f(\overline{W})}$. This means that
$f(\overline{W})\setminus f(bdW)$ is an open set in $X$, which contains $y$ and is contained in $V$. So, $V$ is also open. $\Box$

\section{Proof of Theorem 1.4 and Corollary 1.5}

Let $\widehat{H}_*$ be the exact homology (see \cite{ma}, \cite{sk}),
and $\mathbb Q, \mathbb R$ denote the groups of rational and the real numbers, respectively.
It is well known that for locally compact metric spaces the exact homology is isomorphic with the Steenrod homology.
%Note that for any group $G$ and a closed set $A\subset X$ we have $\widehat{H}_n(X,A;G)\simeq H_n(X,A;G)$ provided $\dim X=n$ \cite{sk}.
For any abelian group $G$ the homological dimension $h\dim_GY$ of a compactum $Y$ is the greatest integer $m$ such that $\widehat{H}_{m}(Y,A;G)\neq 0$ for some closed $A\subset Y$ (if there is no such $m$, then $h\dim_GY=\infty$). It follows from the exact sequence
$$0\to\mathrm{Ext}(H^{m+1}(Y,A),G)\to\widehat{H}_{m}(Y,A;G)\to\mathrm{Hom}(H^{m}(Y,A),G)\to 0$$
that $h\dim_GY\leq\dim Y$. Moreover, by \cite{sk1}, $h\dim_GX$ is the greatest $m$ such that the local homology group $\widehat{H}_m(X,X\setminus x;G)=\varinjlim_{x\in U}\widehat{H}_m(X,X\setminus U;G)$ is not trivial for some $x\in X$.

\textit{Proof of Theorem $1.4$.}
%Suppose $X$ is a homogeneous compact metric $ANR$-space of dimension $n$. According to Theorem 1.1, every point $x\in X$ has a basis $\mathcal B_x=\{U_k\}_{k\geq 1}$ consisting of open sets such that $\overline U_1$ is a proper subset of $X$, $\overline U_{k+1}\subset U_k$, $H^{n-1}(\overline U_k)=H^{n}(\overline U_k)=0$ and $H^{n-1}(bdU_k)$ is a non-trivial finitely generated group (everywhere in this section the coefficient group $\mathbb Z$ in all homology and cohomology groups is suppressed).

$(1)\Rightarrow (2)$. Suppose $X$ is dimensionally full-valued. Then, according to \cite{ha}, $h\dim_{\mathbb Z}X=\dim_{\mathbb Z}X=n$. Hence,
$\widehat{H}_{n}(X,X\setminus x)\neq 0$ for some $x\in X$ (the coefficient group $\mathbb Z$ in all homology and cohomology groups is suppressed).
Because $\dim X=n$, the groups $\widehat{H}_{n}(X,X\setminus x)$ and $H_{n}(X,X\setminus x)$ are isomorphic, see \cite[Theorem 4]{sk}. So,
$H_{n}(X,X\setminus x)\neq 0$.

  $(2)\Rightarrow (3)$. Let $H_{n}(X,X\setminus x)\neq 0$ for some $x\in X$. Then $H_{n}(X,X\setminus U)\neq 0$ for sufficiently small neighborhoods $U$ of $x$ in $X$.
Since by \cite[Theorem 4]{sk} the groups $H_{n}(X,X\setminus U)$ and $\widehat{H}_{n}(X,X\setminus U)$ are isomorphic, $\widehat{H}_{n}(X,X\setminus V)\neq 0$
for some neighborhood $V$ of $x$. On the other hand, $\dim X=n$ implies
 $H^{n+1}(X,X\setminus V)=0$. Hence, it follows from the exact sequence
 $$\mathrm{Ext}(H^{n+1}(X,X\setminus V),\mathbb Z)\to\widehat{H}_{n}(X,X\setminus V)\to\mathrm{Hom}(H^{n}(X,X\setminus V),\mathbb Z)\to 0$$
 that there exists a nontrivial homomorphism from $H^{n}(X,X\setminus V)$ into $\mathbb Z$. This yields that $H^{n}(X,X\setminus V)$ contains elements of
 infinite order. Thus, $H^{n}(X,X\setminus V)\otimes\mathbb Q\neq 0$ and, by the universal coefficients formula, $H^{n}(X,X\setminus V;\mathbb Q)\neq 0$.
 So, $\dim_{\mathbb Q}X=n$. Because $X$ is an $ANR$, we have the following inequalities $\dim_{\mathbb Q}X\leq\dim_{\mathbb S^1}X\leq\dim X$, see Example 1.3(1) and Theorem 12.3(2) from \cite{dr}. Therefore,  $\dim_{\mathbb S^1}X=n$.

$(3)\Rightarrow (1)$. Assume $\dim_{\mathbb S^1}X=n$. The exact sequence $$0\rightarrow\mathbb Z\rightarrow\mathbb R\rightarrow\mathbb S^1\rightarrow 0$$
implies that $\dim_{\mathbb S^1}X\leq\max\{\dim_{\mathbb R}X,\dim X-1\}$, see \cite{dr}. Hence, $\dim_{\mathbb R}X=n$. According to \cite{ha}, both the homological and the cohomological dimensions with respect to any field coincide, so we have
$h\dim_{\mathbb R}X=\dim_{\mathbb R}X=n$. Thus, there exist $x\in X$ and a neighborhood $U$ of $x$ in $X$ such that $\widehat{H}_n(X,X\setminus U;\mathbb R)\neq 0$. As in the proof of the implication $(2)\Rightarrow (3)$, considering the short exact sequence
$$\mathrm{Ext}(H^{n+1}(X,X\setminus U),\mathbb Z)\to\widehat{H}_{n}(X,X\setminus U)\to\mathrm{Hom}(H^{n}(X,X\setminus U),\mathbb Z)\to 0,$$
we can show that $\dim_{\mathbb Q}X=n$. This implies that $X$ is dimensionally full-valued. $\Box$

\textit{Proof of Corollary $1.5$.} Let $X$ be a metric homogeneous $ANR$ compactum with $\dim X=3$. According to \cite[Corollary 2.7]{kru}, we have $\overline{H}_3(X,X\setminus x)\neq 0$, where $\overline{H}_3(X,X\setminus x)$ denotes the singular homology group. On the other hand, by \cite[Lemma 4]{ko1}, the groups
$\overline{H}_3(X,X\setminus x)$ and $H_3(X,X\setminus x)$ are isomorphic. Then, Theorem 1.4 yields that $X$ is dimensionally full-valued. $\Box$

\textbf{Acknowledgments.} The author would like to express his gratitude to A. Dranishnikov, J. Bryant, A. Karassev and G. Skordev for their helpful comments.
The author also thanks the referees and the editors for their valuable remarks and suggestions which improved the paper.

\end{document}